\def\arxiv{}
\ifx\arxiv\undefined 
\documentclass[ECP]{ejpecp}

\SHORTTITLE{Invariance under quasi-isometries of subcritical and supercritical behaviour} 

\TITLE{Invariance under quasi-isometries of subcritical and supercritical behaviour in the
Boolean model of percolation \thanks{The first and the third author were partially supported by FAPESP grants 09/52379-8 and 2012/24086-9, respectively.}}

\AUTHORS{%
 Cristian F. Coletti \footnote{Universidade Federal do ABC
    \EMAIL{cristian.coletti, daniel.miranda@ufabc.edu.br}}
  \and Daniel Miranda \footnotemark[2] \and Filipe Mussini \footnote{Uppsala Universitet \EMAIL{filipe.mussini@math.uu.se}} }


\AMSSUBJ{60G55; 51K05} 

\SUBMITTED{May 1 2015} 
\ACCEPTED{???, 2015} 

\VOLUME{0}
\YEAR{2012}
\PAPERNUM{0}
\DOI{vVOL-PID}

\ABSTRACT{In this work we study the Poisson Boolean model of percolation in locally compact Polish metric spaces
 and we prove the invariance of subcritical and supercritical phases under mm-quasi-isometries. In other words, we prove that  if  the Poisson  Boolean model of percolation is subcritical or supercritical 
(or exhibits phase transition) in a metric space $M$  which is mm-quasi-isometric to a metric space $N$, then these phases also exist for the Poisson Boolean model of percolation in $N$.

Then we apply these results  to understand the phenomenon of phase transition in a 
large family of metric spaces. Indeed, we study the Poisson Boolean model of percolation in the context of Riemannian manifolds,  
in a large family of nilpotent Lie groups and in Cayley graphs. Also, we prove the 
existence of a subcritical phase in Gromov spaces with bounded growth at some scale.}

\usepackage[utf8]{inputenc}

\usepackage{url,xspace}
\usepackage{enumerate}
\usepackage{todonotes}

\newcommand{\bbR}{\mathbb{R}}
\newcommand{\IR}{I\kern -0.37 em R}
\newcommand{\bbH}{\mathbb{H}}
\newcommand{\IH}{I\kern -0.37 em H}

\newcommand{\IL}{I\kern -0.37 em L}
\newcommand{\bbN}{\mathbb{N}}
\newcommand{\bbZ}{\mathbb{Z}}

\newcommand{\cV}{\mathcal{V}}

\renewcommand{\bar}[1]{\overline{#1}}

\newcommand{\qi}{quasi-isometry\xspace}
\newcommand{\mmqi}{mm-quasi-isometry\xspace}

\newcommand{\rank}[1]{\operatorname{rank}\left(#1\right)}

\newcommand{\cl}[1]{\operatorname*{cl}\left(#1\right)}
\newcommand{\inte}[1]{\operatorname*{Int}\left(#1\right)}

\newcommand{\norma}[1]{\lVert#1\rVert}

\renewcommand{\epsilon}{\varepsilon}

\newcommand{\negrito}[1]{\textbf{#1}\index{#1}}

\def\P{{\mathbb{P}}}

\def\cB{{\mathcal B}}

\def\cB{{\mathcal B}}

\def\cX{{\chi}}
\def\cC{{\mathcal C}}

\newcommand{\bmodel}{(\chi, {R})}

\newtheorem{teo}{Theorem}[section]
\newtheorem{corol}[teo]{Corollary}
\newtheorem{ddefi}[teo]{Definition}
  {\begin{shaded} \begin{ddefi}\rm}{\end{ddefi} \end{shaded}}
\newenvironment{defi} 
  {\begin{ddefi}\rm}{\end{ddefi}}
\newtheorem{prop}[teo]{Proposition}

\newtheorem{lema}[teo]{Lemma}
\newtheorem{obss}[teo]{Observation}
\newtheorem{obss1}[teo]{Observations}

\newenvironment{exx}
  {\vspace{0.5cm} \refstepcounter{teo} \noindent  \textbf{Exemplo \arabic{chapter}.\arabic{teo}}}{\vspace{0.2cm} }

\newenvironment{Exemplos} 
  {\addtocounter{teo}{1} \vspace{0.1cm} \noindent \textbf{ Exemplos \arabic{chapter}.\arabic{teo}}}{ \vspace{0.2cm} }

  {\addtocounter{teo}{1} \vspace{0.1cm}\begin{leftbar} \noindent \textbf{ Exemplos \arabic{chapter}.\arabic{teo}}}{\end{leftbar}  \vspace{0.2cm} }
\newenvironment{enumerar}{\begin{list}{\labelitemi}{\leftmargin=2em} }{\end{list}} 

\makeindex
\begin{document}

 \else
 
\documentclass[final,12pt]{elsarticle}
 \usepackage{amsmath,amsthm}
\usepackage{charter}
\usepackage[charter]{mathdesign}

\usepackage[utf8]{inputenc}

\usepackage{url,xspace}
\usepackage{enumerate}
\usepackage{todonotes}
\makeatletter
\def\ps@pprintTitle{%
 \let\@oddhead\@empty
 \let\@evenhead\@empty
 \def\@oddfoot{\centerline{\thepage}}%
 \let\@evenfoot\@oddfoot}
\makeatother
\def\bibsection{\section{References}}
\usepackage[T1]{fontenc}

\setlength{\topmargin}{-0.35in}
\setlength{\textheight}{9.0in}   
\setlength{\textwidth}{6in}    
\setlength{\oddsidemargin}{0.17in}
\setlength{\evensidemargin}{0.17in}
\setlength{\headheight}{25pt}
\setlength{\headsep}{8pt} 
\setlength{\footskip}{0.8cm}

\makeindex
\begin{document}

\title{Invariance under quasi-isometries of subcritical and supercritical behaviour in the
Boolean model of percolation }
\author[ufabc]{Cristian F. Coletti\fnref{fn1}\corref{cor1}}
\ead{cristian.coletti@ufabc.edu.br}

\author[ufabc]{Daniel Miranda}
\ead{daniel.miranda@ufabc.edu.br}

\author[uppssala]{Filipe Mussini\fnref{fn1}}
\ead{filipe.mussini@math.uu.se}

\fntext[fn1]{The first and the third author were partially supported by FAPESP grants 09/52379-8 and 2012/24086-9, respectively}

\address[ufabc]{Universidade Federal do ABC}
\address[uppsala]{Uppsala Universitet}

\begin{abstract}In this work we study the Poisson Boolean model of percolation in locally compact Polish metric spaces
 and we prove the invariance of subcritical and supercritical phases under mm-quasi-isometries. In other words, we prove that  if  the Poisson  Boolean model of percolation is subcritical or supercritical 
(or exhibits phase transition) in a metric space $M$  which is mm-quasi-isometric to a metric space $N$, then these phases also exist for the Poisson Boolean model of percolation in $N$.

Then we apply these results  to understand the phenomenon of phase transition in a 
large family of metric spaces. Indeed, we study the Poisson Boolean model of percolation in the context of Riemannian manifolds,  
in a large family of nilpotent Lie groups and in Cayley graphs. Also, we prove the 
existence of a subcritical phase in Gromov spaces with bounded growth at some scale.
\end{abstract}

\maketitle

\tableofcontents

\fi



\section{Introduction}

The aim of this work is to prove the invariance of subcritical and supercritical phases 
under mm-quasi-isometries. 
This result is in the same spirit of the result about recurrence and transience behaviour of random walks (Polya's Theorem) on graphs which are quasi-isometric to graphs where Polya's Theorem holds. For a precise statement 
of this result see, for instance, \cite{yuval}  and references therein.

The Boolean model of percolation may be  informally described as follow: first 
consider a simple point  process $\cX$ in a locally compact Polish
metric 
space $(M,d)$. Then, at each point of $\cX$, center a ball of random 
radius. Assume that the family of random radius forms an infinite countable collection of  i.i.d. random variable which are also independent of $\cX$.
Thus, the metric space $M$ is partitioned into two 
regions, the occupied region, which is
defined as the union of all random balls, and the vacant region, which is the 
complement of the occupied region.
We say that the Boolean model percolates if the occupied region has 
an unbounded connected component with positive probability.

This definition encompasses two models widely studied in the literature. The first model is the continuum Boolean model of percolation, in which the most prominent examples are the percolation model in the euclidean space $\mathbb{R}^n$ and in the hyperbolic space $\mathbb{H}^n$. The second model is the discrete Boolean model of percolation on graphs.

The history of continuum percolation began in 1961 when W. Gilbert 
\cite{gilbert1961random} introduced the random connection
model on the plane. In $1985$, S. Zuev and A. Sidorenko 
\cite{zuev1985continuous} considered continuum models of percolation where 
points are chosen randomly in space and surrounded by shapes which can be 
random  or fixed. In that work the authors studied the relation between 
critical  parameters associated to that model. In the same year, P. Hall 
\cite{hall1985continuum}  considered the (Poisson) Boolean model of
percolation in $\mathbb{R}^n$.  Denote
by $R$ one of the random radius in this model. P. Hall  proved 
that if the expectation of $R^{2n-1}$ is finite then, almost surely, there is 
no  unbounded component on the occupied region. For a comprehensive study of 
continuum models of percolation,  see the book of R. Meester and R. Roy 
\cite{roy}.

In $2001$, I. Benjamini and O. Schram considered  percolation models  in the 
hyperbolic plane and on regular tilings in the hyperbolic plane. Recently, the 
Boolean model of  percolation has received considerable attention. In 
\cite{tykesson}, J. Tykesson studied the Poisson 
Boolean continuum model of percolation in $\mathbb{H}^n$. He showed that, depending on the intensity of the underlying Poisson
point process, there are none, one or infinitely many unbounded components in the occupied and vacant regions.

In 2008, J. Gou\'{e}r\'{e} \cite{gouere} proved that for the Poisson Boolean model of 
percolation in 
$\mathbb{R}^n$ there is a subcritical phase if, and only if, $R$ has 
finite 
$n$-moment. Here, as before,
$R$ stands for one of the random radius in the model. 

Finally, the discrete version of the Boolean model of percolation was studied   on doubling spaces by C. Coletti, D. Miranda and 
S. Grynberg in 
\cite{cristiangrynbergdaniel}   where they proved the 
existence of a subcritical phase. As 
a byproduct of the subcriticality of the discrete Boolean model of percolation, C. Coletti and S. Grynberg  \cite{cristiangrynberg} constructed, forward in time, interacting particle systems with generator  admitting a Kalikow-type decomposition.

This paper is organized as follows. In section $2$ we give a brief 
introduction to the Boolean model of percolation. In section $3$ we 
define  quasi-isometries and  enunciate  the 
main result of this work concerning the invariance  under quasi-isometries of subcritical and supercritical phases. In section $4$ we provide examples of the use of this result in the study of subcritical and supercritical phases on Cayley graphs, (quasi-periodic) lattices of $\mathbb{R}^n$, Riemannian manifolds and nilpotent Lie groups. Lastly, in section $5$ we prove the main result of this work.

\section{The Boolean Model of Percolation}

A \textbf{ metric and measure space} or a \textbf{mm-space} is a triple $(M,d,\mu )$ where $d$ is a metric such that $M$ is a locally compact Polish metric spaces equipped with a Borel measure $\mu$ defined on $\left(M,\mathcal{B}\left(M\right)\right)$.

Let $\nu$ be a Radom measure on $\left(M,\mathcal{B}\left(M\right)\right)$. A point process on $M$ with intensity measure $\nu$ is defined as follows. Let $\Sigma$ be the set of counting measures on $\left(M, \mathcal{B}\left(M\right)\right)$. It is useful to identify $\Sigma$ with the set of all finite and infinite configurations of points in $M$ without limit points. Let $\mathcal{N}$ be the $\sigma$-algebra on $\Sigma$ generated by the cylinders $\{\nu \in \Sigma | \nu \left(E\right) = k\}$ where $k$ is any integer number and $E$ is any bounded Borel set. Formally, a point process is a measurable mapping $\chi$ from a probability space $\left(\Omega, \mathcal{F}, \mathbb{P}\right)$ into $\left(\Sigma, \mathcal{N}\right)$. We remind the reader that every member of the support of a Poisson point process has multiplicity $1$ under the hypothesis that the corresponding intensity measure assigns the value $0$ to each one-point set. Therefore, if $\chi$ is a Poisson point process on a locally compact Polish metric graph $M$ equipped with a Borel measure $\mu$, then each vertex of $M$ in the support of $\chi$ may have multiplicity greater than one.

Now we define the Boolean model of percolation with balls of fixed radius. First fix $R > 0$. Then to each $x$ in a realization of $\chi$ associate a radius $R$ and form the random ball $B\left(x, R\right)$. Let
$$
\cC := \bigcup_{x \in \chi} B(x,R)\qquad \mbox{ and } \qquad \cV = \cC^c.
$$
This procedure partitions the space into two regions. The random set $\cC$ is called  the \textbf{occupied region} and $\cV$ is called  the \textbf{vacant region}. Each connected 
component in the occupied (vacant) region is called \textbf{occupied (vacant) connected component}. The random set $\cC$ is the Boolean model with parameter $\nu$. The Boolean model is defined on the probability space $\left(\Omega, \mathcal{F}, \mathbb{P}\right)$ defined above. For simplicity we denote the Boolean model just by $\left(\chi, R\right)$.

We say that there is \textbf{percolation} in the occupied 
(vacant) region if there exists an unbounded occupied (vacant) connected component 
with positive probability. We call this event \textbf{percolation event}.

A Boolean model in a mm-space is said to be \textbf{homogeneous} if the intensity  measure $\nu$ of the corresponding point process is a multiple of $\mu$, i.e., $\nu=\lambda\mu$, where $\lambda > 0$. The parameter $\lambda$ is known as the intensity of the point process.

Let $\mathbb{P}_{\lambda, R}(\mbox{Percolation})$ be the probability of the percolation event for the homogeneous Boolean model with intensity $\lambda$ and balls of fixed radius $R$. A standard coupling argument shows that this probability is monotone in $\lambda$. Therefore, we can safely define
 $$
 \lambda_c(R) : = \sup\{\lambda : \mathbb{P}_{\lambda,R}(\mbox{Percolation}) =0\}.
 $$
\begin{defi}
  We say that the homogeneous Boolean model $\left(\chi, R\right)$ exhibits \textbf{phase transition} if there exists $\lambda>0$ such that $\lambda_c(R) >0$.
\end{defi}

\begin{defi}
The homogeneous Boolean model $\left(\chi ,R\right)$ with intensity $\lambda$ and balls of fixed radius $R$ in a mm-space $M$ is \textbf{supercritical} or exhibits a supercritical phase if $\lambda_c\left(R\right) > 0$ and $\lambda_c < \lambda$. Analogously, we say that the homogeneous Boolean model $\left(\chi ,R\right)$ with intensity $\lambda$ and balls of fixed radius $R$ in a mm-space $M$ is \textbf{subcritical} or that exhibits a subcritical phase if $\lambda_c\left(R\right) > 0$ and $0 < \lambda < \lambda_c$. 
\end{defi}  

We finish this section observing that the homogeneous Poisson Boolean model of percolation with intensity $\lambda$ on a metric graph $M$ correspond to the case of Bernoulli Boolean model of percolation on graphs. In the Bernoulli Boolean model of percolation with balls of fixed radius $R$ we center a ball of radius $R$ at each point in a realization of a Bernoulli point process on $M$ with retention parameter $p$. Indeed, a \emph{Bernoulli point process} on $M$ with retention parameters ${\bf{p}} = (p_m)_{m \in M}$, where $0 < p_m < 1$ for any $m \in M$, is a family of independent $\{0,1\}$-valued random variables $\chi=(X_m:\, m\in M)$ such that $\P(X_m=1)=p_m$. The relation between the parameters $\lambda$ and $p$ is as follows. If $\nu$ denotes the intensity measure of $\chi$, then $p_m=1-e^{-\lambda \nu(\{m\})}$ for any $m \in M$.

\section{Invariance}
\subsection{Quasi-Isometries}
The main theme of this paper is the invariance under mm-quasi-isometries of subcritical and supercritical behavior of the Boolean model of percolation. Roughly speaking, a quasi-isometry is a map between  
metric spaces that ignores small-scale structures. We now give a precise definition of quasi-isometry. 

\begin{defi}
Let $\left(M,d_M\right)$ and $\left(N,d_N\right)$ be metric spaces. A map $F: M \to N$ is called a \textbf{quasi-isometry} with parameters $\left(\alpha,\beta,\gamma\right)$ if there exist constants $\alpha \geq 1$ and $\beta,\gamma \geq 0$ such that
 \begin{enumerate}[a)]
  \item  $\forall x,y \in M : \frac{1}{\alpha}d_M(x,y) - \beta\leq d_{N}(F(x),F(y)) \leq \alpha d_M(x,y) +\beta\quad$ (quasi-distance preserving),  
  \item $\forall z \in N, \exists x \in M : d_{N}(z, F(x)) \leq \gamma \quad $ (quasi-surjectivity).
 \end{enumerate}
\end{defi}

It is easy to see thath quasi-isometry is an equivalence relation on metric spaces(\cite{bridson}). Indeed, for each quasi-isometry  $F: M \to N$ there exists a quasi-isometry $G: N \to M$ such that $d_M\left(G\left(F\left(x\right)\right),x\right) \leq \tilde{\gamma}$ and $d_N\left(F\left(G\left(y\right)\right),y\right) \leq \tilde{\gamma}$ for some parameter $\tilde{\gamma}$. The quasi-isometry $G$ is called the 
\textbf{quasi-inverse} of $F$. Moreover, we may choose the same set of parameters for both, the quasi-isometry and the corresponding quasi-inverse, as shown in \cite{shchur}.

We begin by proving some simple geometric facts about the image and preimage of balls under quasi-isometries.

\begin{lema}\label{quasibolamenor}\label{quasibolamaior}
Let $F: M \to N$ be a quasi-isometry with 
parameters $(\alpha,\beta,\gamma)$. Fix $p \in M$ and  $R >0$. Let $B=B(p,R)$ and $\bar{B} = F(B(c,R))$. Then
\begin{enumerate}[a)]

\item  There exists $R' > 0$ such that $\bar{B} \subset B(F(p),R')$,

\item If, for some $k > 0, R > \alpha\beta + 2\alpha k$ then there exists $R'> k$ such that $B(F(p), R')$ $\bigcap F(M) \subset \bar{B}$.

\end{enumerate}
\end{lema}

\begin{proof}
We begin by proving the first part of the lemma. If $y \in \bar{B}$, then there exists $x \in B(p,R)$ such that $F(x)=y$. Since $F$ is a quasi-isometry, then $d_N(F(p),F(x)) \leq \alpha d_M(p,x) +\beta \leq \alpha R+\beta=R'$.  Therefore $\bar{B} \subset B(F(p),R')$ as claimed.

Now we prove the second part of the lemma. First, note that, if $R> \alpha\beta +2\alpha k$, then there exists $R'$ such that $k < R'<2k$ and $R>\alpha(R'+\beta)$. Now let $y$ be a point in $M$ such that $F(y) \in B(F(p), R')$. Then $d_{N}(F(p),F(y))<R'$, 
which implies $d_M(p,y) <\alpha R'+\beta <R$. Consequently, $y \in B(p,r)$ and $F(y) \in \bar{B}$ which is the desired conclusion. 

\end{proof}

The following proposition will be essential for the construction of an induced 
Poisson process in $N$.

\begin{prop}\label{particaoqi}
 Let $F: M \to N$ be a quasi-isometry with parameters $(\alpha,\beta,\gamma)$ between locally compact Polish metric spaces. Then $N$ admits a countable disjoint covering $\{K_i\}$ such that $\cl{K_i}$ is compact. Here $\cl{A}$ stands for the closure of the set $A$. Moreover, the sets $E_i:= \{x \in M : F(x) \in K_i\}$ are disjoints, cover $M$ and have compact closure. 
\end{prop}

\begin{proof}
 Let $F: M \to N$ be a \qi with parameters $(\alpha,\beta,\gamma)$. Then, for each $y \in N$, there exists $x \in M$ such that $d_{N}(y,F(x)) < \gamma$. Now consider the collection of all open balls $\{B(y,\gamma) : y \in N\}$. Since $N$ is a separable and metrizable space, then it satisfies the second enumerability axiom. Therefore, this covering contains a countable subcovering $\{B(y_n,\gamma) : n \in \bbN \}$. Now define
 
 $$
 K_i := B(y_i,\gamma) \setminus \bigcup_{j=1}^{i-1} B(y_j, \gamma).
 $$
 It is easy to check that the sets $K_i$ are disjoints and cover $N$. Since $N$ is locally compact, the closure of each $K_i$ is compact \cite{vaughan}. The properties of the collection of sets $\{E_i : i \in \mathbb{N}\}$ easily follow from the fact that $E_i = F^{-1}(K_i)$.
\end{proof}

\subsubsection{A short description of the construction of Poisson processes} \label{ppconstruction} 

Now we give a short description of how to construct Poisson process with a given intensity measure $\nu$ on locally compact spaces. Let $\left(E_i\right)_{i \in \mathbb{N}}$ be the countable measurable partition of $M$ given by Proposition \ref{particaoqi}. For each $i$, let ${\bf{Y_i}} = \left(Y^i_k:k \in \mathbb{N}\right)$ be an i.i.d. sequence having common distribution $\frac{1}{\nu(E_i)} \nu$, and $N_i$ a Poisson random variable that is independent of the sequence ${\bf{Y_i}}$ and has mean $\nu\left(E_i\right)$. Then make $X_i = \sum_{k=1}^{N_i} \delta_{Y^i_k}$ with the convention that the empty sum equals the zero measure. Then $X = \sum_{i=1}^{\infty} X_i$ is a Poisson point process on $M$ with intensity measure $\nu$. For further details on the construction of Poisson processes in metric spaces see, for instance, \cite{gray}. 

\subsubsection{Measures induced by quasi-isometries and mm quasi-isometries} 

Let $F: M \to N$ be a \qi between metric spaces and let $\mu$ and $\mu'$ be 
Borel  measures on $\left(M,\mathcal{B}\left(M\right)\right)$ and $\left(N,\mathcal{B}\left(N\right)\right)$, respectively. 
Define $\mu^*(K_i) = 
\mu(F^{-1}(K_i)) = \mu (E_i)$, where the sets $K_i$ and $E_i$ are the ones given 
by Proposition \ref{particaoqi}. We extend the definition of $\mu^*$ to 

\begin{equation}\label{muestrela}
\mu^*(D) = \mu'(D)\displaystyle \sum_{i: D \cap K_i \neq \emptyset} \frac{\mu^*(K_i)}{\mu'(K_i)},
\end{equation}
where $D \in \cB(N)$.
It is easy to check that $\mu^*$ is a measure in $N$. We call $\mu^*$ the \negrito{quasi-isometry induced measure}. We observe that the measure $\mu^*$ depends on the particular covering $\{K_i\}$. In the set up of this work this does not represent a problem since the measure $\mu^*$ will be dominated from above and from below by the intensities measures of two point processes which will be enough for our purposes. For more details see Section $5$.

\begin{defi}
 Let $F:M \to N$ be a \qi between mm-spaces and let $\mu$ and $\mu'$ be two measures defined in $\mathcal{B}\left(M\right)$ and $\mathcal{B}\left(N\right)$ respectively. Let $\{K_i\}$ be the covering of $N$ given by Proposition \ref{particaoqi}. We say that the \qi is a \negrito{\mmqi} or that it is \negrito{ measure compatible} if there exist strictly positive constants $C_1, C_2,C_3, C_4$ such that, for any $i$, $C_1 < \mu'(K_i) < C_2$ and $C_3 < \mu^*(K_i) < C_4$.
\end{defi}

Henceforth, $\mu$ and $\mu'$ will always denote  measures defined in $\mathcal{B}\left(M\right)$ and $\mathcal{B}\left(N\right)$ respectively.

\subsection{Main Result}
Now we are ready to state the main result of this paper.
\begin{teo}\label{teomaisimp}
 Let $F: M \to N$ be a mm-quasi-isometry with parameters 
$(\alpha,\beta,\gamma)$. Let $\bmodel$ be a homogeneous Poisson
Boolean model in $M$ driven by a homogeneous Poisson point process with intensity $\lambda$.

\begin{enumerate}[a)]
\item Assume that $\bmodel$ is subcritical and that $R > \alpha\beta+2\alpha\gamma$. Then, there exist $\lambda' > 0$ and $R' > 0$ such that the homogeneous Poisson Boolean model $(\chi', R')$ in $N$ driven by a homogeneous Poisson point process $\chi'$ with intensity $\lambda'$ is subcritical .

\item Assume that $\bmodel$ is supercritical. Then, there exist $\lambda' > 0$ and $R' > 0$ such that the homogeneous Poisson Boolean model $(\chi', R')$ in $N$ driven by a homogeneous Poisson point process $\chi'$ with intensity $\lambda'$ is supercritical.
\end{enumerate}

\end{teo}

The proof of Theorem \ref{teomaisimp} will be given in section \ref{provadoteomaisimp}. The following corollaries follows immediately from Theorem \ref{teomaisimp} and their proofs are omitted.

\begin{corol}\label{corolimp1}
Let $F: M \to N$ be a mm-quasi-isometry with parameters 
$(\alpha,\beta,\gamma)$. Let $\bmodel$ be a homogeneous Poisson  Boolean model in $M$ with $R > \alpha\beta+2\alpha\gamma$ exhibiting phase 
transition.  Then there exist $\lambda > 0$ and $R' > 0$ such that the homogeneous Poisson Boolean model $(\chi', R')$ in $N$ driven by a homogeneous Poisson point process $\chi'$ with intensity $\lambda$ exhibites phase transition.
\end{corol}

Henceforth, for simplicity whenever we say "for sufficiently large radius" we mean $R > \alpha\beta+2\alpha\gamma$ where $(\alpha,\beta,\gamma)$ are
the parameters of a mm-quasi-isometry $F: M \to N$ between metric spaces.

\begin{corol}\label{corolimp2}
For sufficiently large radius, the existence of subcritical and supercritical phases and phase transition in the homogeneous Poisson  Boolean model is invariant under mm-quasi-isometries admitting a mm-quasi-inverse.
\end{corol}

\section{Applications}
Before we prove Theorem \ref{teomaisimp}, we give some examples to which this theorem apply in order to understand the phase transition phenomenon in some  families of metric spaces. First, we prove that the existence of phase transition for the Boolean model of percolation on Cayley graphs of finitely generated groups does not depend on the choice of the generating set. Secondly, we prove that the Boolean model is subcritical (supercritical) in a Riemannian manifold $M$ if and only if the Boolean model is subcritical (supercritical) on the graph obtained trough a discretization procedure of $M$. As a corollary of this result, we prove the existence of phase transition in a large family of nilpotent Lie groups, tilings, etc. Finally, we prove the  existence of a subcritical phase in Gromov  spaces with bounded growth at some scale.

\subsection{Cayley Graphs}

Let $H$ be a finitely generated group and let $S$ be a symmetric generating set.  
The \negrito{Cayley Graph} $C(H,S)$ is the graph whose set of vertices is $H$ and there is an edge joining $u$ and $v$ if and only if $u =  sv$.

 Given a finitely generated group $H$ and a finite generating set $S$, we define the \negrito{word norm} by
$$
\norma{g}_S := \inf_n\{s_i \in S : s_1 \dots s_n = g\}
$$
and the \negrito{word metric} by $d_S(g,h): = \norma{g^{-1}h}_S.$ 

It is clear that the geometry of a  Cayley graph depends on the choice of the generating set. However, as we will show below, the existence of a subcritical or supercritical phase does not depend on that choice. Indeed, next theorem will be essential for this purpose.

\begin{teo}[\cite{bridson}] \label{cayleygraph}
 Let $H$ be a finitely generated group and let $S$ and $S'$ be two distinct generating sets. Then the Cayley graphs $C(H,S)$ and $C(H,S')$ are quasi-isometric.
\end{teo}

The following theorem allows, when restricted to uniformly bounded graphs such as $C(H,S)$, to  remove the measure compatible hypothesis in Theorem \ref{teomaisimp}.

\begin{teo}\label{discretofunciona}
Let $F: M \to N$ be a \qi between uniformly bounded graphs equipped with the counting measure on $\left(M,\mathcal{P}\left(M\right)\right)$ and $\left(N,\mathcal{P}\left(N\right)\right)$ respectively. Then $F$ is \mmqi.
\end{teo}

\begin{proof}
It suffices to verify the measure compatibility condition for one vertex of $C$. Firstly, we note that $1 \leq \mu'(x) \leq n$, where $n$ is the maximum number of incident edges in $x$. Since $F$ is a quasi-isometry, then $F^{-1}(\{x\})$ is a bounded set in $M$ and, by Lemma 
\ref{quasibolamaior}, we have that there exists a ball of radius $A+B$ such that 
$F^{-1}(\{x\}) \subset B(x, A+B)$. Since $M$ is a uniformly bounded graph, we have that $1 < 
\mu(F^{-1}(\{x\})) < \dfrac{n^d -1}{n +1}$, where $d$ is the smallest integer greater than $A+B$. Thus, $F$ is a  \mmqi and the proof is complete.
\end{proof}

\begin{teo}
 The existence of phase transition  of a homogeneous Poisson  Boolean model on a Cayley graph does not depend on the choice of its generating set, for sufficiently large radius.
\end{teo}

\begin{proof}
 Let $H$ be a finitely generated countable group and let $S$ and $S'$ be two distinct generating set of $H$. Let $C(H,S)$ and $ C(H,S')$ be the Cayley graphs associated with the generating sets $S$ and $S'$ respectively. It follows from Theorem \ref{discretofunciona} that the Cayley graphs $G$ and $G'$ are mm-quasi-isometric. Hence, the result follows from Theorem \ref{teomaisimp}.
\end{proof}

\subsection{Riemannian Manifolds}
Before embarking on the implications of Theorem \ref{teomaisimp} in the context of Riemannian manifolds we give some definitions.

Let $M$ be a complete and connected Riemannian Manifold. A subset $S$ of $M$ 
is said to be \negrito{$\varepsilon$-separated} if $d(x,y)\geq\varepsilon>0$ for all distinct 
$x,y \in S$. 

A \negrito{discretization} of $M$ is a graph $\Gamma$ whose vertex set is 
given by  a $\varepsilon$-separated subset $S$ of $M$ for which there exists $\rho>0$ such that $M=\bigcup_{c\in S}B(c,\rho)$. We say that $\varepsilon$ is the \negrito{separation} and $\rho$ is the \negrito{discretizations cover radius}. 

The edge set of the graph $\Gamma$ is given by the family of all pair of  neighbors in $S$ where the set of neighbors of a vertex $c \in S$ is $N(c):=\{S\bigcap \cl{B(c,2\rho)}\setminus\{c\}\}$.

\begin{teo}[\cite{isochavel}]\label{qichavel}
 Let $M$ be a complete and connected Riemannian manifold with Ricci curvature bounded from bellow and let $S$ be a  $\varepsilon$-separated set. Then the associated graph $\Gamma$ is uniformly bounded and with the graph metric is quasi-isometric to $M$. 
\end{teo}

\begin{teo}[\cite{isochavel}]\label{chavelbaixo}
 Let $M$ be a complete Riemannian manifold, with Ricci curvature bounded from below and assume that there exist constants $r_0$ and  $V_0$ such that 
\[
\mu(B(x,r_0)) \geq V_0,
\]
for any $x \in M$. Then, for all $r >0$, there exists $C(r)>0$ such that 
$$
\mu(B(x,r)) \geq C(r)
$$
for any $x \in M$.
\end{teo}

\begin{teo}\label{variedadefunciona}
Let $M$ be a Riemannian manifold with Ricci curvature bounded from below which admits a discretization. Let $\mu$ be the volume measure on $\left(M, \mathcal{B}\left(M\right)\right)$ and let $\Gamma$ be the graph associated to its discretization. Assume the existence of nonnegative constants $r_0$ and $V_0$ such that, for all $x \in M, \mu(B(x,r_0)) \geq V_0$.   Then there exists a \mmqi $F: M \to \Gamma$.
\end{teo}

\begin{proof}
The existence of a \qi  $F:M \to \Gamma$ is guaranteed by the previous theorem. Therefore, we only need to prove that this \mmqi  is measure compatible.

\noindent It follows from the Bishop-Gromov's Theorem that $\mu(B(x,C)) \leq \mu_E(B(x,C))$, where $\mu_E$ is the Euclidean volume of the ball which depends only on the radius of the ball. Let $C_2 = \mu_E(B(x,C))$ be the volume of the ball of radius $C$. Then  $\mu(B(x,C)) \leq C_2$. Since $K_i \subset B(x,C)$ for some $x$, we have $\mu(K_i) \leq C_2$.
For the lower bound just note that, by the graph construction, the points of 
$S$ are sufficiently distant so that the sets $K_i$ have not empty interior. In 
this way, we have a ball 
$B\left(x_i, \frac{\varepsilon}{2}\right) \subset K_i$ by \cite{clarkson}. By Theorem \ref{chavelbaixo}, we have $\mu(B(x,r)) \geq C(r)$, for all $x\in M$. Then  $C_1 = C(\frac{\varepsilon}{2}) \leq \mu(B(x,\frac{\varepsilon}{2})) \leq \mu(K_i)$, and therefore $C_1 \leq \mu(K_i) \leq C_2$. Since $\Gamma$ is a uniformly bounded graph, the rest of the proof is analogous to the proof of Theorem \ref{discretofunciona}. 
\end{proof}

Therefore, $M$ is quasi-isometric to any of its discretizations and any two discretizations of $M$ are quasi-isometric between them. As a consequence of Theorems \ref{teomaisimp} and the previous theorem we have

 \begin{teo}
 Let $M$ be a Riemannian  manifold with Ricci curvature bounded from below which admits a discretization. Let $\mu$ be the volume measure on $(M,\mathcal{B}(M))$ and let $\Gamma$ be the graph associated with its discretization. Assume also the existence of non zero positive constants $r_0$ and $V_0$ such that $\mu(B(x,r_0)) \geq V_0$.
  Then, for sufficiently large radius, the Boolean model of percolation in a Riemannian manifold $M$ is subcritical (supercritical) if and only if the Boolean model of percolation on $\Gamma$ is subcritical (supercritical).
\end{teo}

\paragraph{Group Actions}

Let $M$ be a given complete Riemannian Manifold and let $\Gamma$ be a finitely generated isometry subgroup of $M$ acting freely and properly discontinuously. Namely, for any $p \in M$, there exists a neighborhood $U$ of $p$ such that $U \cap g(U) =\emptyset$ for all $g \neq e$. 

Now we recall the \v{S}varc-Milnor's Theorem.

\begin{teo}[Proposition 8.19 in \cite{bridson}]
Let $M$ be a complete Riemannian manifold and let $\Gamma$ be a finitely generated isometry subgroup of $M$ acting freely and properly discontinuously such that the quotient  $M/\Gamma$ is a compact manifold. Then $M$ and $\Gamma$ are quasi-isometric.
\end{teo}

Since the quotient $M/\Gamma$ is a compact manifold, a standard compactness argument shows that  $M$ is a manifold with Ricci curvature bounded from below for which there exist strictly positive constants $r_0$ and $V_0$ such that $\mu'(B(x,r_0)) \geq V_0$, see \cite{manfredao} for more details. Hence, the proof of the following theorem follows directly from the \v{S}varc-Milnor's Theorem and Theorems \ref{teomaisimp} and \ref{variedadefunciona}. 

\begin{teo}
Let $M$ be a complete Riemannian manifold and let $\Gamma$ be a finitely generated isometry subgroup of $M$ acting freely and properly discontinuously such that the quotient $M/\Gamma$ is a compact manifold. Then, the Boolean model in $M$ is subcritical (supercritical) if and only if the Boolean model is subcritical (supercritical) on $\Gamma$, for sufficiently large radius. 
\end{teo}

\paragraph{Tilings}

 A \textbf{tiling} of a manifold $M$ is a  family $\{K_1, \dots, K_n\}$ of compact sets where $K_i \subset M$  and a family $\phi_i^j$ of isometries, where $1 \leq i \leq n$, such that
$$
\bbR^n = \bigcup_{i,j} \phi_i^j(K_j) \mbox{ and } 
\mu(\phi_i^j(K_i)\bigcap\phi_s^l(K_s)) = 0.
$$

We may associate a graph $\Gamma$ with each tiling as follows. For each $K_i$, choose a point $p_i \in \inte{K_i}$. Let the vertex set be the set of all points $\phi_i^j(p_i)$. By convenience, we will denote the points that are image of an isometry by
$p_i^j = \phi_i^j(p_i)$ and $p_i^0 = p_i$. 
Now say that a pair $\{p_i^j , p_k^l\}$ is an edge 
if and only if $K_i^j \cap K_k^l \neq \emptyset$. Then, in a similar way to Theorem \ref{qichavel}, we have

\begin{corol}

Let $M$ be a Riemannian manifold, $\mathcal{K}$ be a tiling of $M$ and $\Gamma$ be the graph given by this tiling. Then $M$ is  quasi-isometric to $\Gamma$.
\end{corol}

\noindent As a consequence of Theorems \ref{discretofunciona} and \ref{teomaisimp}, we get the following result.
 
\begin{teo}
Let $M$ be a Riemannian manifold, $\mathcal{K}$ be a tiling of $M$ and $\Gamma$ the graph given by this tiling. Then, for sufficiently large radius, the Boolean model of percolation is subcritical on $\Gamma$  if and only if the Boolean model is subcritical in $M$. Moreover, the Boolean model of percolation is supercritical on $\Gamma$ if and only if the Boolean model is supercritical in $M$.
\end{teo}

We emphasize that the previous theorem also applies  to aperiodic tiling, i.e. tiling that does not have translational symmetry such as the Penrose tilings.

%
%

 \paragraph{Nilpotent Lie Groups}

\begin{defi}
Let $G$ be a locally compact group. A discrete subgroup $H$ of $G$ is called a
\negrito{lattice} in $G$ if the quotient $G/H$ has a finite invariant measure.
\end{defi}

Now we recall  some classical results about finite generated nilpotent groups and about lattices in nilpotent Lie groups.

\begin{teo}\cite{rag}
 Let $N$ be a simply connected and nilpotent Lie group. Denote by $\mathfrak{n}$ its Lie algebra. Then $N$ has a lattice if and only if $\mathfrak{n}$ has a basis with rational structure constants.
\end{teo}

\begin{teo}\cite{gromov}[Gromov Theorem] A finitely generated group have polynomial growth if and only if it has a nilpotent subgroup with finite index. 
\end{teo}

\begin{teo}\cite{bass}\cite{guivarc}[Bass-Guivarc'h Formula]
The order of the polynomial growth of $G$ is 
 $$d(G) = \sum_{k\geq1} k \rank{G_k/G_{k+1}}$$
\end{teo}

\begin{teo}[Corollary 7.18 of \cite{yuval}]\label{coryuval}
 If $G$ is a Cayley graph of a group $\Gamma$ with, at most, polynomial growth, then either $\Gamma$ is isomorphic to $\bbZ$ or there is phase transition in $G$.
\end{teo}

\begin{teo}

Let $G$ be a nilpotent Lie group of dimension at least $2$ such that its Lie algebra has rational structure constants. Then, for sufficiently large radius, the Boolean model in $G$ has phase transition.
\end{teo}

\begin{proof}
Since $G$ is a Lie group such that its associated Lie algebra has rational structure constants, then there exists a lattice $H \subset G$. Therefore, we get from Theorem \ref{qichavel} that $G$ is quasi-isometric to $H$. It follows from Gromov's Theorem that $H$ has polynomial growth. From the Bass-Givarc'h identity we get that $H$ has growth of at least $2$. Since $d(\bbZ)=1$, it follows from Theorem \ref{coryuval} that there is phase transition in $H$. Since $H$ is mm-quasi-isometric to $G$, we have that there is phase transition in $G$ as a consequence of Theorems \ref{teomaisimp} and \ref{variedadefunciona}. 
\end{proof}

\subsection{Gromov Spaces with bounded growth at some scale}

Now we prove the subcriticality for the Boolean model of percolation in Gromov spaces of bounded growth at some scale using  a quasi-isometry between a Gromov spaces  and  a convex subset of the hyperbolic space. For more details about this quasi-isometry see \cite{bonkemb}. 

\begin{defi} 
 We say that a metric space $M$ is  a \textbf{geodesic metric space} if any two points in $M$ are the extremes of a minimizing geodesic segment. 
A \textbf{geodesic triangle} is  the set  consisting of three different points together with the pairwise-connecting geodesic lines.  Given $ \delta > 0 $, we say that a geodesic triangle is $\delta $-slim if any one of its sides is contained in $\delta$-neighborhood  of the union of  the other two sides. A 
geodesic metric is  a Gromov Space or $ \delta $-hyperbolic if 
all its triangles are $ \delta $-slim for some $ \delta> 0 $. 
\end{defi} 

\begin{figure}[h]
 \begin{center}
 \includegraphics[width=3.6cm]{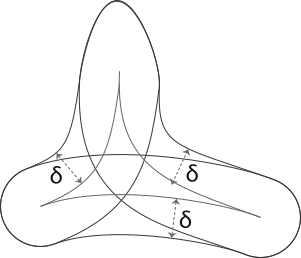}
  \caption{$\delta$-slim Triangle.}
\end{center}

\end{figure}

As examples of Gromov  space we cite trees, Cayley graphs associated with the hyperbolic groups (this family includes almost all Cayley graphs, see \cite{gromov1987hyperbolic} for a precise meaning of this assertion) and Riemannian manifolds with negative scalar curvature.

A metric space has \negrito{bounded growth at some scale} if there exist constants 
 $R>r>0$ e $N \in \bbN$ such that every ball of radius $R$ 
can be covered by at most $N$ balls of radius $r$.

Note that, for graphs, bounded degree implies bounded growth at some scale, hence the Cayley graphs of every finitely generated hyperbolic group is a Gromov space with  bounded growth at some scale.

\begin{teo}[\cite{bonkemb}]\label{schramm}
Let $M$  be a Gromov space with bounded growth at some scale.
Then there exists $n$ such that $M$ is quasi-isometric to a convex subset of  $\bbH^n$. 
\end{teo}

Consider a Gromov space $M $  with bounded growth at some scale.
We will show that the Boolean model of percolation in $M$ is subcritical. It follows from Theorem \ref{schramm} that there exists a \qi $ F: M \to Y \subset \bbH ^ n $, where $Y$ is a convex set. Without loss of generality, we can assume that $ Y $ has non-empty interior. For otherwise, since $Y$ is convex, it would be contained in a lower dimensional hyperbolic space (see \cite{bridson}).

We can also assume that the dimension of $Y$ is at least $2$. Otherwise, $Y$ would be quasi isometric to a subset of $ \bbR $ and in this case the problem of subcriticality is trivial.

Thus, we may assume that $M$ is quasi-isometric to a bounded subset $Y$ of $\mathbb{H}^n$ with non empty interior and dimension at least $2$.

\begin{teo}
Let $M$  be a Gromov space with bounded growth at some scale.
Then, for sufficiently large radius, the Boolean model of percolation in $M$ is subcritical.
\end{teo}

\begin{proof}
Let $(\chi,R)$ be the homogeneous Poisson Boolean model in $\bbH^n$ driven by a homogeneous Poisson point process $\chi$ in $\mathbb{H}^n$ with intensity $\lambda$. It follows from Theorems $4.2$ and $4.3$ in \cite{tykesson} that for sufficiently large radius there exists $\lambda_c$ such that for $\lambda < \lambda_c$ the homogeneous Boolean model is subcritical. Therefore, we assume that the intensity of $\chi$ is $\lambda < \lambda_c$ and that the radius of the random balls is sufficiently large.

Let $Y$ be the bounded subset of $\mathbb{H}^n$ whcih is quasi-isometric to $M$. Denote by $\Sigma(Y)$ the Boolean model induced on $Y$, i.e. $\Sigma(Y)=\bigcup_{y \in Y} B(y, R)$. Let $\mathcal{C}(Y) $ and $ \mathcal{C}(\bbH^n) $ denote the connected occupied components for the Boolean models in $Y$ and $ \bbH^n$ respectively.
Since there is no unbounded connected component in $ \bbH ^ n $,
then there is no unbounded connected component in $ Y $. Therefore, the Boolean model $\Sigma (Y)$ does not percolate. By the previous theorem, there exists a \qi $ F: Y \to M $. Since the volumes of any ball depends only on it radius (see \cite{ratcliffe} ), it follows from Theorem \ref{variedadefunciona} that $F$ satisfies the hypotheses of Theorem \ref{teomaisimp}. Then we may conclude that the Boolean model of percolation in $M$ does not percolate. This finishes the proof.
\end{proof}

\section{Proof of the Main Result}
\label{provadoteomaisimp}

The proof of Theorem \ref{teomaisimp} is based on a coupling between a Poisson Boolean model in $N$ with intensity  measure  $\mu^*$ and the image under a mm-quasi-isometry of the occupied connected components of a Poisson Boolean model in $M$. Then we consider a version of the homogeneous Poisson Boolean model in $N$ and we arrive at the conclusion of Theorem \ref{teomaisimp} by a suitable comparison between the Poisson Boolean model with bounded intensity and the homogeneous Boolean model in $N$.

\begin{defi} Let $F: M \rightarrow N$ be a quasi-isometry and let $\mu$ be a measure on $\left(M,\mathcal{B}\left(M\right)\right)$. Let $(\chi,R)$ be the Boolean model in $M$ with balls of fixed radius $R$. Then
 \begin{enumerate}[a)]
  \item the \negrito{induced Boolean model} $(\chi',R')$ is
the Boolean model with balls of fixed radius $R'$ in $N$ 
where the center of the random balls belongs to $F \left(\chi \right)$;
 \item the \negrito{induced Poisson Boolean model with balls of fixed radius $R'$}  is the 
Boolean model  in $N$ driven by a Poisson point process with intensity  measure  $\mu^*$ and with balls of fixed radius $R'$.
 \end{enumerate}
\end{defi}

\begin{prop}
\label{inducedBooleanmodel}
(Percolation in the induced Boolean model)
 Let $F: M \rightarrow N$ be a quasi-isometry with parameters $\left(\alpha, \beta, \gamma \right)$ and  let $\bmodel$ be the Boolean model in $M$ with balls of fixed radius $R$. Then
  \begin{enumerate}[a)]
  \item if $\bmodel$ percolates, then there exists $R' > 0$ such that the induced Boolean model 
$(\chi', R')$ in $N$ percolates;
  \item if $\bmodel$ does not 
percolate  and $R > \alpha\beta+2\alpha\gamma$,  then there exists $R' > 0$ such that the induced Boolean model $(\chi', 
R')$ in $N$ does not percolate.
  \end{enumerate}
\end{prop}

\begin{proof} We begin by proving the first part of the proposition. Assume that the Boolean model $\bmodel$ in $M$ percolates. 
Let $(\chi', R')$ be the induced Boolean model in $N$ where the existence of $R'$ is guaranteed by part $a)$ of Lemma 
\ref{quasibolamaior}. Then the image of the Boolean model $\bmodel$ under $F$ is contained in 
the union of all the balls in the Boolean model $(\chi', R')$. This implies that the image 
of any unbounded connected component in the Boolean model $\bmodel$ is contained in 
a subset of the union of all balls in the induced Boolean model. Since $F$ is a quasi-isometry, we may conclude that the induced Boolean model percolates.

Now we prove the second part of the proposition. Assume that the Boolean model $\bmodel$ in $M$ 
does not percolate. Let $(\chi', R')$ be the induced Boolean model in $N$ where $R'$ is given by part $b)$ of Lemma \ref{quasibolamenor}. Then, every ball in the Boolean model $(\chi', R')$ is contained in the image of the Boolean model $\bmodel$. This implies that all 
the connected components of the induced Boolean model are contained in the 
image of the connected components of the Boolean model in $M$. Since there is no unbounded connected component in $M$ and $F$ is a quasi-isometry, then there is no unbounded connected component in the induced Boolean model in $N$. This completes the proof

\end{proof}

\begin{teo}(Percolation in the induced Poisson Boolean model) \label{existqi}
 Let $F: M \to N$ be a quasi-isometry with parameters $(\alpha,\beta,\gamma)$.  Let 
$(\chi, R)$ be the Boolean model in $M$ with balls of fixed radius $R$. Then
\begin{enumerate}[a)]

\item If $(\chi, R)$  percolates , then  there exists $R'$ such that the induced Poisson Boolean model in $N$ with balls of fixed radius $R'$ percolates. \item If $(\chi,R)$  does not percolate and if
 $R>\alpha\beta+2\alpha\gamma$, then there exists $R'$ such that the induced Poisson Boolean model in $N$ with balls of fixed radius $R'$ does not percolate.
\end{enumerate}
\end{teo}

\begin{proof}
Since $\mu(E_i) = \mu(F^{-1}(K_i)) = \mu^*(K_i)$, we may assume, without loss of generality, that the Poisson processes driving the Boolean model on $M$ and the Poisson induced Boolean model on $N$ are such that the number of random points in $E_i \subset M$ and in $K_i \subset N$ are a.s. the same. For this purpose it suffices to consider versions of these Poisson processes which use the same family of Poisson random variables used in the construction of the Poisson point process on $M$ with intensity measure $\mu$ and the Poisson point process on $N$ with intensity measure $\mu^*$ (see subsection \ref{ppconstruction} for more details). 
 \begin{enumerate}[a)]

  \item By Proposition \ref{inducedBooleanmodel}, there exists $R' > 0$ such that the induced Boolean model $(\chi', R')$ percolates. Since we may assume that the number of points of $\chi$ in $E_i$ and of $\chi'$ in $K_i$ are  a.s. the same, then each point $F(p)$ of the induced point process will be at a distance no greater than $\gamma$ of some point $p^*$ in the induced Poisson point process. Therefore, $B(F(p), R') \subset B(p^*, R'+ \gamma)$ and  $\bigcup B(F(p), R') \subset \bigcup B(p^*, R'+\gamma)$. Since the induced process percolates, then the induced model will also percolate.
  \item 
  
  Assume that the Boolean model $(\chi,R)$ does not percolate in $M$. By Proposition \ref{inducedBooleanmodel} we have that there exists $R'$ such that the induced Boolean model $(\chi', R')$ does not percolate. By Lemma \ref{quasibolamenor} we can assume that $R'>\gamma$. It follows from the discussion above and from the quasi-surjectivity of the quasi-isometry $F$ that any point $F(p)$ in the induced Boolean model is at a distance no greater than $\gamma$ of some point $p^*$ in the induced Poisson Boolean model. Then for every point $p \in \chi$, there exists a.s. a point $p^*$ in the Poisson induced point process such that $d_{N}(F(p),p^*) < \gamma$ which implies $B(p^*, R' -\gamma) \subset B(F(p), R')$ and  $\bigcup B(p^*, R' -\gamma) \subset \bigcup B(F(p), R')$. Since the induced Boolean model does not percolate, then the Poisson induced Boolean model does not percolate too.
 \end{enumerate}

\end{proof}

\begin{lema}\label{intmono}
 Let $F: M \to N$ be a \qi with parameters  $\left(\alpha, \beta, \gamma \right)$.
 Let $\mu_1$ and $\mu_2$ be two measures defined in $\mathcal{B}\left(M\right)$ and let $\mu'$ be a measure defined in $\mathcal{B}\left(N\right)$. Assume that $\mu_1(E) \leq \mu_2(E)$ for any bounded set $E \in \cB(M)$. Then the \qi induced measure satisfies the inequality $\mu_1^*\left(D\right) \leq \mu_2^*\left(D\right)$ for any $D \in \mathcal{B}\left(N\right)$.
\end{lema}

\begin{proof}
 Let $D \in \cB(N)$. Since $\mu_1^*\left(K_i\right)=\mu_1\left(F^{-1}\left(K_i\right)\right) \leq \mu_2\left(F^{-1}\left(K_i\right)\right) = \mu_2^*\left(K_i\right)$, we get 

$$
 \mu_1^*\left(D\right) = \mu'(D) \displaystyle \sum_{i : D \cap K_i \neq \emptyset} \frac{\mu_1^*(K_i)}{\mu'(K_i)} \leq \mu'(D) \displaystyle \sum_{i: D \cap K_i \neq \emptyset} \frac{\mu_2^*(K_i)}{\mu'(K_i)} = \mu_2^*(D),
$$ 
which is the desired conclusion.
\end{proof}

 \begin{lema}\label{linhaestrela}
 Let $F: M \to N$ be a \qi with parameters $\left(\alpha, \beta, \gamma \right)$ . Let $\mu$ and $\mu'$ be two measures defined in $\mathcal{B}\left(M\right)$ and $\mathcal{B}\left(N\right)$ respectively. Also, let $\mu^*$ be the quasi-isometry induced measure. Assume that there exist strictly positive constants $C_1,C_2,C_3,C_4$ such that $C_1 < \mu'(K_i) < C_2$ and $C_3 < \mu^*(K_i) < C_4$ for any $i$. Then there exist strictly positive constants $\bar{C}_1$ and $\bar{C_2}$ such that
$$
\bar{C}_1 \mu'(D) < \mu^*(D) < \bar{C}_2 \mu'(D)
$$
for any bounded set $D \in \mathcal{B}(N)$.
\end{lema}

\begin{proof}
Let $D \in \mathcal{B}\left(N\right)$ be a bounded set. It follows from the definition of $\mu^*$ that $\mu'\left(K_j\right) \mu^*\left(D \cap K_i\right) =\mu'\left(D \cap K_i\right) \mu^*\left(K_i\right)$. Since $D$ is bounded and $\{i : D \cap  K_i \neq \emptyset\}$ is finite, we have
$$
\mu^*\left(D\right) = \displaystyle \sum_{i:D \cap K_i \neq \emptyset} \mu^*\left(D \cap K_i\right) < \frac{C_4}{C_1} \displaystyle \sum_{i:D \cap K_j \neq \emptyset} \mu'\left(D \cap K_i\right) = \frac{C_4}{C_1} \mu'\left(D\right).
$$

Analogously, we may show that $\frac{C_3}{C_2} \mu'\left(D\right) \leq \mu^*\left(D\right)$. Then make $\bar{C}_1 = \frac{C_3}{C_2}$ and $\bar{C}_2 = \frac{C_4}{C_1}$, and the proof is complete.
\end{proof}
 
 Now we show that we can construct a homogeneous point process using a point process with bounded intensity measure.

\begin{defi}
 We say that a point process $\chi$ with intensity measure $\mu_\Lambda$ has 
\negrito{bounded intensity with respect to $\mu$} if, for any bounded set $E \in \mathcal{B}\left(M\right)$, there exist strictly positive constants $\lambda_1, \lambda_2$  such that 
$\lambda_1\mu(E)< \mu_\Lambda(E) <\lambda_2\mu(E)$, where $\mu_\Lambda(E) = \int_E \Lambda d\mu$. 
\end{defi}
When the measure is clear from the context, we will only say that the point process has bounded intensity.

\begin{prop}\label{proposicaomuestrela}
 Let $F: M \to N$ be a \mmqi  and let $\mu$ and $\mu'$ be two measures defined in $\mathcal{B}\left(M\right)$ and $\mathcal{B}\left(N\right)$ respectively. Let $\chi$ be a Poisson point process in $M$ with bounded intensity. Let $D \in \mathcal{B}\left(N\right)$ be any bounded set. Then there exist strictly positive constants $\bar{\lambda_1}$ and $\bar{\lambda_2}$ such that 
 $$
 \bar{\lambda_1}\mu'(D) < \mu^*_\Lambda(D) < \bar{\lambda_2}\mu'(D).
 $$
\end{prop}

\begin{proof}
 Let $E \in \mathcal{B}\left(M\right)$ be a bounded set. Since $\chi$ has bounded intensity, there exist strictly positive constants $\lambda_1$ and $\lambda_2$ such that $\lambda_1\mu(E) < \mu_\Lambda(E) < \lambda_2 \mu(E)$. By Lemma \ref{intmono}, we have that $\lambda_1\mu^*(D) < \mu^*_\Lambda(D) < \lambda_2\mu^*(D)$ for any bounded set $D \in \mathcal{B}\left(N\right)$. Since $F$ is a \mmqi, then there exist strictly positive constants $C_1, C_2,C_3, C_4$ such that, for any $i$, $C_1 < \mu'(K_i) < C_2$ and $C_3 < \mu^*(K_i) < C_4$. It follows from Lemma \ref{linhaestrela}, that $\bar{C_1}\mu'(D) <\mu^*(D) < \bar{C_2}\mu'(D)$ for some strictly positive constants $\bar{C_1}$ and $\bar{C_2}$. This implies $\bar{C_1}\lambda_1\mu'(D) <\mu^*_\Lambda(D) < \bar{C_2}\lambda_2\mu'(D)$. Now make $\bar{\lambda_1} = \bar{C_1}\lambda_1, \bar{\lambda_2} = \bar{C_2} \lambda_2$ and the proof is complete.
\end{proof}

\begin{prop}(Homogenization Theorem)\label{homoqi}
 Let $(\chi,R)$ be a Poisson Boolean model in $M$ with bounded intensity measure $\mu_\Lambda$. Then
 \begin{enumerate}[a)]
\item if the Poisson Boolean model does not percolate, then there exists $\lambda_1 > 0$ such that the Poisson Boolean model $(\hat{\chi},R)$ driven by the homogeneous Poisson point process $\hat{\chi}$ with intensity $\lambda_1$ does not percolate;
  
 \item if the Poisson Boolean model percolates, then there exists $\lambda_2 > 0$ such that the Poisson Boolean model $(\hat{\chi},R)$ driven by the homogeneous Poisson point process $\hat{\chi}$ with intensity $\lambda_2$ percolates.
\end{enumerate}  
  
\end{prop}
\begin{proof}

We begin by proving the first part of the proposition. Since $(\chi,R)$ has bounded intensity, then there exist $\lambda_1, \lambda_2$ such that $\lambda_1\mu(E)< \mu_\Lambda(E) <\lambda_2\mu(E)$, where $\mu_\Lambda(E) = \int_E \Lambda d\mu$. Then we may obtain a version of the Poisson point process with intensity measure $\lambda_1\mu$ by a $p_1$-thinning of the Poisson point process	with intensity measure $\lambda_2 \mu$ where $p_1:M\rightarrow [0,1]$ is given by $p_1(m)= \lambda_1/\lambda_2$. Analogously, we may obtain a version of the Poisson point process with intensity measure $\mu_{\Lambda}$ by a $p_2$-thinning of the Poisson point process with intensity measure $\lambda_2 \mu$ where $p_2:M \rightarrow [0,1]$ is given by $p_2(m) = \Lambda(m)/\lambda_2$.

Then the set of center of the random balls (with fixed radius $R$) of the Boolean model driven by a homogeneous Poisson process with intensity $\lambda_1$ are a.s. a subset of the center of the random balls of the Poisson Boolean model (with fixed radius $R$) driven by a Poisson process with intensity measure $\mu_{\Lambda}$. Therefore, if the Boolean model $(\chi,R)$ with bounded intensity measure $\mu_\Lambda$ does not percolate then the Boolean model driven by a homogeneous Poisson point process does not percolate too.

The proof for the supercritical case follows in the same lines of the proof above.

\end{proof}

\subsection*{Proof of Theorem \ref{teomaisimp}} 

\noindent Let $F: M \to N$ be  a mm-quasi-isometry with parameters $(\alpha,\beta,\gamma)$ and let $\bmodel$ be the homogeneous Boolean model in $M$ with intensity $\lambda$.
Assume that the Boolean model does not percolate. Since $R > \alpha \beta + 2 \alpha \gamma$, Theorem \ref{existqi} guarantees the existence of $R' > 0$ such that the (nonhomogeneous) induced Poisson Boolean model with intensity measure $\mu^\star$ does not percolate in $N$. It follows from proposition \ref{proposicaomuestrela} and the fact that $F$ is a \mmqi that $\mu^*$ is bounded with respect to $\mu'$.  It follows from Proposition \ref{homoqi} that we can homogenize the induced Poisson Boolean model into a Boolean model $(\chi ', R')$ that does not percolate in $N$. This completes the proof for the subcritical case.

The proof for the supercritical case follows in the same lines of the proof above.

$\hfill \square$

\section*{Acknowledgments} We thank Luiz Renato Fontes and Glauco Valle for invaluable contributions to a previous version of this work. We also thank Rafael Grisi for many fruitful discussions about this work.


\bibliographystyle{amsplain}
\bibliography{referencias}{}

\end{document}